\documentclass[12pt]{amsart}

\usepackage[colorlinks]{hyperref}
\usepackage{fullpage}
\usepackage{stmaryrd}
\usepackage{graphicx}
\usepackage{hyperref}
\usepackage{amsmath,amssymb,amsfonts,mathrsfs,amscd}
\usepackage{newtxtext,newtxmath}
\usepackage[all,cmtip]{xy}

\newtheorem{thm}{Theorem}[section]
\newtheorem{lem}[thm]{Lemma}
\newtheorem{cor}[thm]{Corollary}

\newtheorem{thmA}{Theorem}

\theoremstyle{definition}

\theoremstyle{remark}

\numberwithin{equation}{section}

\newcommand{\NM}{\vartriangleleft}

\newcommand{\X}{\mathfrak{X}}

\DeclareMathOperator{\Irr}{Irr}

\DeclareMathOperator{\Syl}{Syl}

\DeclareMathOperator{\dz}{dz}
\DeclareMathOperator{\rdz}{rdz}

\begin{document}

\title{Character triples and relative defect zero characters}

\author{Junwei Zhang}
\address{School of Mathematical Sciences, Shanxi University, Taiyuan, 030006, China.}
\email{zhangjunwei@sxu.edu.cn}
\email{jinping@sxu.edu.cn}

\author{Lizhong Wang}
\address{School of Mathematics, Peking University, Beijing, 100871, China.}
\email{lwang@math.pku.edu.cn}

\author{Ping Jin*}

\thanks{*Corresponding author}


\date{}

\maketitle

\begin{abstract}
Given a character triple $(G,N,\theta)$, which means that $G$ is a finite group with $N\NM G$ and $\theta\in\Irr(N)$ is $G$-invariant, we introduce the notion of a $\pi$-quasi extension of $\theta$ to $G$
where $\pi$ is the set of primes dividing the order of the cohomology element $[\theta]_{G/N}\in H^2(G/N,\mathbb{C}^\times)$ associated with the character triple, and then establish the uniqueness of such an extension in the normalized case.
As an application, we use the $\pi$-quasi extension of $\theta$ to construct a bijection from the set of $\pi$-defect zero characters of $G/N$ onto the set of relative $\pi$-defect zero characters of $G$ over $\theta$.
Our results generalize the related theorems of M. Murai and of G. Navarro.
\end{abstract}

\section{Introduction}

Let $G$ be a finite group, $N$ a normal subgroup of $G$ and $\theta\in\Irr(N)$ a $G$-invariant irreducible complex character of $N$.
In this situation, we say that $(G,N,\theta)$ is a \emph{character triple},
and it is known that this triple is associated with a cohomology element $[\theta]_{G/N}\in H^2(G/N,\mathbb{C}^\times)$,
which can be used to determine whether or not $\theta$ is extendible to $G$.
Specifically, $\theta$ extends to $G$ if and only if  $[\theta]_{G/N}$ is trivial,
and this happens precisely when $\theta$ extends to $P$ for every Sylow subgroup $P/N$ of $G/N$.
(See, for example, Theorem 11.7 and Corollary 11.31 of \cite{I1976}).

To measure the extent to which $\theta$ extends to $G$, it is natural to investigate
the set $\pi$ of primes, which consists of those primes $p$ with the property that $\theta$
is not extendible to $P$ for some $P/N\in\Syl_p(G/N)$.
Of course, the prime set $\pi$ is not empty if and only if $\theta$ does not extend to $G$.
It is easy to show that $\pi$ is exactly the set of prime divisors of the order of $[\theta]_{G/N}$,
and furthermore, we will construct a special class function of $G$ which behaves
(in some respects, at least) as if it is an extension of $\theta$ to $G$,
as indicated in the following.

\begin{thmA}
Let $(G,N,\theta)$ be a character triple, and let $\pi$ be a set of primes.
Then the following are equivalent.

{\rm (a)} The order of $[\theta]_{G/N}\in H^2(G/N,\mathbb{C}^\times)$ is a $\pi$-number.

{\rm (b)}  For each prime $p\notin\pi$, $\theta$ extends to $P$ for some $P/N\in\Syl_{p}(G/N)$.

{\rm (c)} There exists a complex-valued class function $\tilde\theta$ of $G$ such that
for each $\pi'$-subgroup $H/N$ of $G/N$,
the restriction $\tilde\theta_H$ of $\tilde\theta$ to $H$ is an extension (as a character) of $\theta$ to $H$.
\end{thmA}

In the situation of Theorem A, by a powerful theorem of Schmid (Theorem B of \cite{S1985}),
we know that the order of $[\theta]_{G/N}$ always divides each of ${\hat o}(\theta)$,
$\theta(1)o(\theta)$ and $|N|/\theta(1)$. Here ${\hat o}(\theta)$ is the number of distinct roots of unity
contained in $\mathbb{Q}(\theta)$ and $o(\theta)$ is the determinantal order of $\theta$. In particular,
if $|N|/\theta(1)$ is a $\pi$-number, then the order of $[\theta]_{G/N}$ is necessarily a $\pi$-number,
so that Theorem A applies.

Furthermore, we call the class function $\tilde\theta$ in Theorem A(c) a \emph{$\pi$-quasi extension} of $\theta$ to $G$.
Observe that if $H/N$ is a $\pi'$-subgroup of $G/N$, then the $\pi$-part $x_\pi$ of each element $x\in H$ lies in $N$,
so it does not matter which values of $\tilde\theta(x)$ will take for all $x\in G$ with $x_\pi\notin N$.
To establish the uniqueness of $\pi$-quasi extensions, therefore,  we need to modify
the values of $\tilde\theta$.
We say that a $\pi$-quasi extension of $\theta$ to $G$ is \emph{normalized}
if it vanishes on elements $x$ of $G$ with $x_\pi\notin N$,
and we prove that the normalized $\pi$-quasi extension, if exists, is unique up to multiplication by linear characters.

\begin{thmA}\label{unique}
Let $(G,N,\theta)$ be a character triple, and suppose that $\varphi$ and $\varphi'$ are normalized $\pi$-quasi extensions of $\theta$ to $G$. Then there exists a linear character $\lambda$ of $G/N$ such that $\varphi'=\lambda\varphi$.
\end{thmA}

We continue to assume that $(G,N,\theta)$ is a character triple and $\pi$ is a set of primes.
Given a character $\chi\in\Irr(G)$ that lies over $\theta$, we say that $\chi$ has \emph{relative $\pi$-defect zero} over $\theta$
if $$(\chi(1)/\theta(1))_\pi=|G/N|_\pi,$$
and we write $\rdz_\pi(G|\theta)$ for the set of relative $\pi$-defect zero characters of $G$ over $\theta$.
Also, a character $\psi\in\Irr(G)$ has \emph{$\pi$-defect zero} if $\psi(1)_\pi=|G|_\pi$,
which is equivalent to saying that $\psi$ has $p$-defect zero for each prime $p\in\pi$,
and we use $\dz_\pi(G)$ to denote the set of $\pi$-defect zero characters of $G$.
Note that if $\theta=1_N$ is the trivial character of $N$, then $\dz_\pi(G/N)=\rdz_\pi(G|1_N)$.
We mention that (relative) $p$-defect zero characters play an important role when studying the Alperin weight conjecture;
see Theorems 2.4 and 5.1 of \cite{NT2011} for instance.

As an application of Theorem A, we prove the following.
\begin{thmA}
Let $(G,N,\theta)$ be a character triple,
and suppose that $\tilde\theta$ is a $\pi$-quasi extension of $\theta$ to $G$,
where $\pi$ is a set of primes.
Then the map $\chi\mapsto \tilde{\theta}\chi$ defines a bijection from $\dz_\pi(G/N)$ onto $\rdz_\pi(G|\theta)$. 
\end{thmA}

The following result is an essential ingredient of the proof of Theorem C.
Also, we can use it to generalize Theorem 2.3
of \cite{CL2012} (with $\pi$ in place of the prime $p$),
which is enough to obtain the $\pi$-analogues of the main theorems of that paper
(with Isaacs' $\pi$-partial characters in place of $p$-Brauer characters),
but we will not present the details here.

\begin{thmA}
Let $(G,N,\theta)$ be a character triple. If the order of $[\theta]_{G/N}\in H^2(G/N,\mathbb{C}^\times)$ is a $\pi$-number
for a prime set $\pi$, then $|\dz_\pi(G/N)|=|\rdz_\pi(G|\theta)|$.
\end{thmA}

We remark that our theorems generalize those of Murai in \cite{M2012},
but his proofs involved block theory while ours are purely character-theoretic.

All groups considered in this paper are finite, and the notation and terminology are mostly taken from Isaacs' book \cite{I1976}.

\section{Proof of Theorem A}

We begin by reviewing some basic properties of projective representations
and the cohomology element associated with a character triple.

\begin{lem}\label{proj}
Let $(G,N,\theta)$ be a character triple. Then there exists a projective representation
$\mathfrak X$ of $G$ such that the following hold.

{\rm (a)} ${\mathfrak X}_N$ is a representation of $N$ affording $\theta$.

{\rm (b)} $\X(ag)=\X(a)\X(g)$ and $\X(ga)=\X(g)\X(a)$ for all $g\in G$ and $a\in N$.\\
Furthermore, if $\alpha$ is the factor set of $\X$,
then $\alpha\in Z^2(G,\mathbb{C}^\times)$ is the inflation of
$\bar\alpha\in Z^2(G/N,\mathbb{C}^\times)$ and the cohomology element
$$[\theta]_{G/N}=[\bar\alpha]\in H^2(G/N,\mathbb{C}^\times)$$
depends only on $\theta$. Also, $\theta$ is extendible to $G$ if and only if $[\theta]_{G/N}=1$.
\end{lem}
\begin{proof}
This is Theorems 11.2 and 11.7 of \cite{I1976}.
\end{proof}

\begin{lem}\label{res}
Let $(G,N,\theta)$ be a character triple with cohomology element $[\theta]_{G/N}$
defined in Lemma \ref{proj}, and suppose that $H/N$ is a subgroup of $G/N$. Then the following hold.

{\rm (a)} The restriction map ${\rm res}^{G/N}_{H/N}: H^2(G/N,\mathbb{C}^\times)\to  H^2(H/N,\mathbb{C}^\times)$
carries $[\theta]_{G/N}$ to $[\theta]_{H/N}$, the cohomology element associated with the character triple $(H,N,\theta)$.

{\rm (b)} If $\theta$ extends to $H$, then the order of $[\theta]_{G/N}$ divides $|G:H|$.
\end{lem}
\begin{proof}
(a) is clear from Lemma \ref{proj}. For (b), we see that $[\theta]_{H/N}=1$,
and the result follows by considering the composition of the restriction map ${\rm res}^{G/N}_{H/N}$ in (a)
with the corestriction map
$${\rm cor}^{G/N}_{H/N}: H^2(H/N,\mathbb{C}^\times)\to  H^2(G/N,\mathbb{C}^\times)$$
(see for instance Proposition 9.2 of \cite{B1982}).
\end{proof}

The following is well known, and we use it to establish some notation.

\begin{lem}\label{triple}
Suppose that $(G,N,\theta)$ is a character triple and let $\mathfrak X$ be a
projective representation of $G$ associated with $\theta$ such that
the factor set $\alpha$ of $\mathfrak X$ takes $n$-th roots of unity values.
Let $Z\le \mathbb{C}^\times$ be the cyclic group of order $n$.
Let $\hat{G}=\{(g,z)|g\in G,\ z\in Z\}$, where we define the multiplication
$$(g_1,z_1)(g_2,z_2)=(g_1g_2, \alpha(g_1,g_2)z_1z_2).$$
Identify $N$ and $Z$ with $N\times 1$ and $1\times Z$, respectively.
Then the following hold.

{\rm (a)} $\hat G$ is a finite group, and the map
$\rho:\hat G\to G$ given by $\rho(g,z)=g$ is a surjective homomorphism with kernel $Z$.

{\rm (b)} $N\NM \hat{G}$, $Z\le Z(\hat G)$,
and $\hat N=N\times Z\NM \hat G$.
Also, $\hat N/N\le Z(\hat G/N)$.

{\rm (c)} The map defined by $\hat{\mathfrak{X}}(g,z)=z\mathfrak{X}(g)$
is an irreducible representation of $\hat G$ whose character $\tau\in\Irr(\hat G)$
extends $\theta$.

{\rm (d)} Let $\hat{\theta}=\theta\times 1_{Z}\in\Irr(\hat{N})$
and $\hat\lambda\in\Irr(\hat N)$ be defined by
$\hat\lambda(a,z)=z^{-1}$.
Then $\hat\theta$ is $\hat G$-invariant,
$\hat\lambda$ is a $\hat G$-invariant linear character with kernel $N$,
and $\hat\lambda^{-1}\hat\theta$ extends to $\tau\in\Irr(\hat G)$.
In particular, $\tau_Z=\theta(1)(\hat\lambda^{-1})_Z$, that is, $\tau$ lies over $(\hat\lambda^{-1})_Z\in\Irr(Z)$.

{\rm (e)} $(\hat G/N,\hat N/N,\hat\lambda)$ is a character triple isomorphic to
$(G,N,\theta)$, where $\hat\lambda$ is viewed as a faithful linear character of $\hat N/N$.
\end{lem}
\begin{proof}
See Theorem 5.6 and the proof of Corollary 5.9 of \cite{N2018}.
\end{proof}

For the definition of character triple isomorphisms, see Definition 11.23 of \cite{I1976}.

\begin{cor}\label{tri}
Let $(G,N,\theta)$ be a character triple, and let $n$ be the order of $[\theta]_{G/N}\in H^2(G/N,\mathbb{C}^\times)$.
Then there exists an isomorphic character triple $(G^*,N^*,\theta^*)$
such that $N^*\le Z(G^*)$ is cyclic of order $n$ and $\theta^*$ is faithful.
\end{cor}
\begin{proof}
Using the notation of Lemmas \ref{proj}, we have $[\bar\alpha]^n=1$, so $\bar\alpha^n\in B^2(G/N,\mathbb{C}^\times)$,
and it follows that $\bar\alpha(\bar x,\bar y)^n=\nu(\bar x)\nu(\bar y)\nu(\bar x \bar y)^{-1}$ for some function $\nu:G/N\to\mathbb{C}^\times$.

Define $\mu:G/N\to\mathbb{C}^\times$ by setting $\mu(\bar x)^n=\nu(\bar x)^{-1}$ for each $\bar x\in G/N$,
and let
$$\bar\beta(\bar x,\bar y)=\bar\alpha(\bar x,\bar y)\mu(\bar x)\mu(\bar y)\mu(\bar x\bar y)^{-1}.$$
Then $[\theta]_{G/N}=[\bar\alpha]=[\bar\beta]$ with $\bar\beta(\bar x,\bar y)^n=1$.
Let $\beta$ be the inflation of $\bar\beta$ to $G\times G$,
so that $\beta(x,y)=\bar\beta(\bar x,\bar y)$ for $x,y\in G$.
Since $\beta$ takes $n$-th roots of unity values,
the result follows by Lemma \ref{triple}(e) (with $\beta$ in place of $\alpha$)
by setting $(G^*,N^*,\theta^*)=(\hat G/N,\hat N/N,\hat\lambda)$.
\end{proof}

We can now prove Theorem A, which we restate as follows in a slightly different form.

\begin{thm}
Let $(G,N,\theta)$ be a character triple and let $\pi$ be a set of primes.
Then the following are equivalent.

{\rm (a)} $\theta$ has a $\pi$-quasi extension to $G$.

{\rm (b)} $\theta$ extends to each subgroup $H$ of $G$, where $N\le H$ and $H/N$ is a $\pi'$-group.

{\rm (c)} $\theta$ extends to $P$, where $P/N\in\Syl_{p}(G/N)$ for each $p\notin\pi$.

{\rm (d)} The order of $[\theta]_{G/N}\in H^2(G/N,\mathbb{C}^\times)$ is a $\pi$-number.
\end{thm}
\begin{proof}
That both (a) implies (b) and (b) implies (c) are trivial, and by Lemma \ref{res}(b) we see that (c) implies (d).
Now assume (d), and we want to prove (a).
By Lemma \ref{proj} and the proof of Lemma \ref{tri},
we can choose the projective representation $\X$ of $G$ such that the values of
the associated factor set $\alpha$ are $n$-th roots of unity, where $n$ is the order of $[\theta]_{G/N}$.

We use the notation of Lemma \ref{triple}.
For $g\in G$, let $x\in\hat G$ be a preimage of $g$
under the homomorphism $\rho:\hat G\to G$, and define
$$\tilde{\theta}(g)=
\left\{
\begin{array}{cc}
\hat\lambda(x_\pi)\tau(x), & \text{if}\, g_\pi \in N,\\
0, & \text{otherwise}.
\end{array}
\right.$$
Note that if $g_\pi \in N$, then $x_\pi \in \hat N$ since $\rho(x_\pi)=g_\pi$ and $\rho^{-1}(N)=\hat N$,
and thus $\hat\lambda(x_\pi)$ makes sense.
Furthermore, if $y\in\hat G$ is another preimage of $g$, then $y=zx$ for some $z\in Z$, and hence
$$\hat\lambda(y_\pi)\tau(y)=\hat\lambda(z)\hat\lambda(x_\pi)\hat\lambda(z)^{-1}\tau(x)=
\hat\lambda(x_\pi)\tau(x),$$
where the first equality holds because $\hat\lambda$ is a linear character of $\hat N$
and $\tau_Z=\theta(1)(\hat\lambda^{-1})_Z$. Thus $\tilde{\theta}$ is a well-defined class function on $G$.

To prove that $\tilde{\theta}$ is a $\pi$-quasi extension of $\theta$ to $G$,
let $H/N$ be any $\pi'$-subgroup of $G/N$, and we want to show that $\tilde\theta_H$ is
a character of $H$ that extends $\theta$.
By Lemma \ref{res}(a), we see that $[\theta]_{H/N}\in H^2(H/N,\mathbb{C}^\times)$ has $\pi$-order,
and by Lemma \ref{proj}, we know that $\theta$ extends to $H$.
Now, fix an extension $\theta^*\in\Irr(H)$ of $\theta$ to $H$,
and let $\hat H=\rho^{-1}(H)$ be the full preimage of $H$ in $\hat G$ via the homomorphism $\rho$.
Then we can view $\theta^*\in\Irr(\hat H)$ that extends $\hat\theta$ and hence $\theta$.
Since $\tau_{\hat H}$ is also an extension of $\theta$ to $\hat H$,
Gallagher' theorem (see Corollary 6.17 of \cite{I1976}) guarantees that $\theta^*=\delta\tau_{\hat H}$
for a unique linear character $\delta\in\Irr(\hat H/N)$. In this case, we have $$\theta(1)1_Z=(\theta^*)_Z=\delta_Z\tau_Z=\theta(1)\delta_Z({\hat\lambda}^{-1})_Z,$$
so that $\delta_Z=\hat\lambda_Z$.

Define $\hat{\mu}(x)=\delta(x)^{-1}\hat\lambda({x}_{\pi})$ for each $x\in \hat H$.
Since $\hat H/\hat N\cong H/N$ is a $\pi'$-group, we see that $x_\pi\in \hat N$ and thus $\hat\lambda(x_\pi)$ makes sense.
We want to show that $\hat{\mu}$ is a linear character of $\hat{H}$.
To prove this, observe that $\hat N/N$ is a central Hall $\pi$-subgroup of $\hat H/N$,
so $N(xy)_{\pi}=N{x}_{\pi}N{y}_{\pi}\in \hat H/N$ for all $x,y\in\hat H$,
and since $\hat\lambda$ is trivial on $N$, we have $\hat\lambda(x_{\pi})\hat\lambda(y_{\pi})=\hat\lambda((xy)_{\pi})$.
Thus $\hat{\mu}(xy)=\hat{\mu}(x)\hat{\mu}(y)$, so that $\hat{\mu}$ is a linear character of $\hat H$, as wanted.

Furthermore, note that both $\delta$ and $\hat\lambda$ are trivial on $N$ and that $\delta_Z=\hat\lambda_Z$.
It follows that $\hat N=NZ$ is contained in the kernel of $\hat\mu$,
and thus we can view $\hat\mu\in\Irr(\hat H/\hat N)$.
Let $\mu\in\Irr(H/N)$ correspond to $\hat\mu$ via the isomorphism $\hat H/\hat N\to H/N$ induced by $\rho:\hat G\to G$.
Then $\mu$ is a linear character of $H/N$, and for each $g\in H$, we choose a preimage $x\in\hat H$ of $g$,
so that $\mu(g)=\hat\mu(x)$. Then we have
$$\mu(g)\theta^*(g)=\hat\mu(x)\theta^*(x)=\delta(x)^{-1}\hat\lambda(x_\pi)\delta(x)\tau(x)
=\hat\lambda(x_\pi)\tau(x)=\tilde\theta(g),$$
and hence $\tilde\theta_H=\mu\theta^{*}$.
This proves that $\tilde\theta_H$ is a character of $H$ that is an extension of $\theta$ to $H$,
so $\tilde\theta$ is indeed a $\pi$-quasi extension of $\theta$ to $G$.
This proves (a), as wanted.
\end{proof}

\section{Proof of Theorem B}
Fix a character triple $(G,N,\theta)$,
and let $\tilde\theta$ be a $\pi$-quasi extension of $\theta$ to $G$.
Then we can obtain a normalized $\pi$-quasi extension $\tilde\theta_n$ of $\theta$ from $\tilde\theta$ by setting
$$\tilde\theta_n(x)=\left\{
\begin{array}{cc}
\tilde\theta(x),  & \text{if}\; x_{\pi}\in N,\\
0,& \text{otherwise.}
\end{array}
\right.$$
Also, in the situation of Theorem C from the introduction,
it is easy to see that $\tilde\theta\chi=\tilde\theta_n\chi$ for all $\chi\in\dz_\pi(G/N)$.
Therefore, we need only consider normalized $\pi$-quasi extensions.

The following is Theorem B, which we restate here for convenience.

\begin{thm}\label{unique}
Let $(G,N,\theta)$ be a character triple, and suppose that $\varphi$ and $\varphi'$ are normalized $\pi$-quasi extensions of $\theta$ to $G$. Then there exists a linear character $\lambda$ of $G/N$ such that $\varphi'=\lambda\varphi$.
\end{thm}
\begin{proof}
By definition, it suffices to show that there is a linear character $\lambda$ of $G/N$ such that $\varphi'(g)=\lambda(Ng)\varphi(g)$
for each $g\in G$ with $g_\pi\in N$.

For each $\pi'$-element $x\in G/N$, let $N_x$ be the full preimage of the cyclic subgroup $\langle x\rangle$ in $G$,
so that $N_x/N\cong \langle x\rangle$ is a $\pi'$-group.
By definition, we see that both $\varphi_{N_x}$ and $\varphi'_{N_x}$ are extensions of $\theta$ to $N_x$,
and thus by Gallagher's theorem, there exists a unique linear character $\mu_x\in\Irr(N_x/N)$ such that $\varphi'_{N_x}=\mu_x\varphi_{N_x}$. We define $\hat\mu(x)=\mu_x(x)$,
so that $\hat\mu$ is a complex-valued function defined on $\pi'$-elements of $G/N$.

Note that if $H/N$ is a $\pi'$-subgroup of $G/N$,
then there is a unique linear character $\nu\in\Irr(H/N)$ such that $\varphi'_H=\nu\varphi_H$
because $\varphi_H$ and $\varphi'_H$ are extensions of $\theta$ to $H$,
and in this case, the uniqueness guarantees that $\nu(x)=\mu_x(x)=\hat\mu(x)$ for each $x\in H/N$.
In particular, $\hat\mu_{H/N}=\nu$,
so the restriction of $\hat\mu$ to each $\pi'$-subgroup of $G/N$ is a linear character.

We claim that $\hat\mu$ is a class function on the set of $\pi'$-elements of $G/N$.
To see this, observe that $\varphi$ and $\varphi'$ are class functions on $G$,
and thus $\mu_{x^y}=(\mu_x)^y$ for all $y\in G/N$. It follows that
$$\hat\mu(x^y)=\mu_{x^y}(x^y)=(\mu_x)^y(x^y)=\mu_x(x)=\hat\mu(x),$$
as desired.

Now, for each $y\in G/N$, we define $\lambda(y)=\hat\mu(y_{\pi'})$.
Then $\lambda$ is clearly a class function on $G/N$.
We must prove that $\lambda$ is a generalized character.
Let $E$ be a nilpotent subgroup of $G/N$ and let $\eta\in\Irr(E)$.
Then we can write $E=A\times B$ and $\eta=\alpha\times\beta$ with $\alpha\in\Irr(A)$ and $\beta\in\Irr(B)$,
where $A$ is a $\pi$-subgroup and $B$ is a $\pi'$-subgroup. Thus we have
$$[\lambda_E,\eta]=\frac{1}{|E|}\sum_{y\in E}\hat\mu(y_{\pi'})\overline{\eta(y)}
=\frac{1}{|A|}\sum_{a\in A}\overline{\alpha(a)}
\frac{1}{|B|}\sum_{b\in B}\hat\mu(b)\overline{\beta(b)}
=[1_A,\alpha][\hat\mu_B,\beta],$$
which is an integer because we have shown that $\hat\mu_B$ is a linear character for each $\pi'$-subgroup $B$ of $G/N$.
This proves that $\lambda$ is a generalized character of $G/N$, as wanted.
Also, we see that $\lambda(1)=1$ and that $\hat\mu(x)$ is a root of unity for all $\pi'$-elements $x$ of $G/N$, so
$$[\lambda,\lambda]=\frac{1}{|G/N|}\sum_{y\in G/N}|\hat\mu(y_{\pi'})|^2=1,$$
which shows that $\lambda$ is a linear character of $G/N$.

Finally, for each $g\in G$ with $g_\pi\in N$, writing $x=Ng=Ng_{\pi'}$,
we have
$$\varphi'(g)=\varphi'_{N_x}(g)=\mu_x(x)\varphi_{N_x}(g)=\hat\mu(x)\varphi(g)=\hat\mu(Ng_{\pi'})\varphi(g)
=\lambda(Ng)\varphi(g),$$
and the result follows.
\end{proof}

In the case where $N$ is a $\pi$-group, we can construct a normalized $\pi$-quasi extension of $\theta$ to $G$
in a similar way used in \cite{N2000} (with the prime set $\pi$ in place of the prime $p$).
For each $g\in G$ with $g_\pi\in N$, since $N\langle g\rangle/N$ is a $\pi'$-group,
Gallagher's theorem guarantees that $\theta$ has a unique extension $\theta_g\in\Irr(N\langle g\rangle)$
such that $o(\theta)=o(\theta_g)$, that is, $\theta_g$ is the canonical extension of $\theta$. We define $\hat\theta(g)=\theta_g(g)$,
and if necessary, we let $\hat\theta(x)=0$ for $x\in G$ with $x_\pi\notin N$,
so that $\hat\theta$ is a complex-valued function defined on $G$.

The following generalizes Lemma 4.1 of \cite{N2000}.

\begin{thm}
Let $(G,N,\theta)$ be a character triple with $N$ a $\pi$-group.

{\rm (a)} $\hat\theta$ is a class function of $G$.

{\rm (b)} If $H/N$ is a $\pi'$-subgroup of $G/N$, then $\hat\theta_H$ is the canonical extension of $\theta$.
In particular, $\hat\theta$ is a normalized $\pi$-quasi extension of $\theta$ to $G$.

{\rm (c)} $\hat\theta$ is never zero on each $\pi'$-element of $G$.
\end{thm}
\begin{proof}
(a)  It suffices to show that $\hat\theta$ is a class function of $G^0=\{g\in G|g_\pi\in N\}$.
Let $x,y\in G^0$ be such that $y=x^g$ for some $g\in G$.
Then $(\theta_x)^g$ is also a canonical extension of $\theta$ to $N\langle y\rangle$,
and the uniqueness of such an extension guarantees that $\theta_{y}=(\theta_{x})^{g}$.
Thus we have
$$\hat{\theta}(y)=\theta_{x^g}(x^g)=(\theta_{x})^{g}(x^{g})=\theta_{x}(x)=\hat{\theta}(x),$$
and the result follows.

(b) Let $\varphi\in\Irr(H)$ be the canonical extension of $\theta$ to $H$,
and we need to prove that $\hat{\theta}_H=\varphi$.
Fix an arbitrary element $x\in H$, and let $X=N\langle x\rangle$.
Then $\varphi_X$ is clearly the canonical extension of $\theta$ to $X$,
and the uniqueness implies that $\theta_x=\varphi_X$.
It follows that $\hat\theta(x)=\theta_x(x)=\varphi(x)$, and we have $\hat\theta_H=\varphi$,
as needed.

(c) Let $g\in G$ be a $\pi'$-element, and write $H=N\langle g\rangle$.
Then $H/N$ is a $\pi'$-group,
and by (b), we see that $\hat\theta_H$ is the canonical extension of $\theta$ to $H$.
In this case, it is well known that $\hat\theta(g)$ cannot be zero (see for instance \cite[Theorem 13.6]{I1976}).
\end{proof}

\section{Proof of Theorem D}

We begin with a useful result.

\begin{lem}\label{alg-int}
Let $(G,N,\theta)$ be a character triple, and let $\chi\in\Irr(G|\theta)$.
If $x\in G$, then
$$\frac{\theta(1)|G:C_G(Nx)|\chi(x)}{\chi(1)}$$
is an algebraic integer, where $C_G(Nx)=\{g\in G|[g,x]\in N\}$, the full preimage of $C_{G/N}(Nx)$ in $G$.
\end{lem}
\begin{proof}
This is precisely Theorem 5.19 of \cite{N2018}.
\end{proof}

We need a generalization of Problem 8.12 of \cite{I1976}.
Actually, what we need in the proof of Theorem D is only Theorem \ref{re}(c) below,
but we prove more for possible use in future.

\begin{thm}\label{re}
Let $(G,N,\theta)$ be a character triple with $N$ a $\pi$-group for a set $\pi$ of primes,
and let $\chi\in\rdz_\pi(G|\mu)$, where $\mu\in\Irr(N)$ is $G$-invariant.
Define the function $\psi$ on $G$ by
$$\psi(g)=\left\{
\begin{array}{cc}
\theta(g_{\pi})\chi(g_{\pi'}), & \text{if}\; g_{\pi}\in N,\\
0,& \text{otherwise}.
\end{array}
\right.$$
Then the following hold.

{\rm (a)} $\psi$ is a generalized character of $G$.

{\rm (b)} If $G=NC_G(N)$ and $\chi\in\dz_\pi(G/N)$,
then $\psi\in\Irr(G)$, and in particular, $\psi\in\rdz_\pi(G|\theta)$.

{\rm (c)} If $N\le Z(G)$, then $\chi\mapsto\psi$ defines a bijection
$\rdz_\pi(G|\mu)\to \rdz_\pi(G|\theta)$.
\end{thm}

\begin{proof}
(a) Let $E$ be a nilpotent subgroup of $G$, and let $\xi\in\Irr(E)$.
By Brauer' generalized character theorem (see Theorem 8.4 of \cite{I1976}), it suffices to show that $[\psi_{E},\xi]\in \mathbb{Z}$.

To do this, we can write $E=A\times B$ where $A$ is a $\pi$-group and $B$ is a $\pi'$-group,
so that $\xi=\alpha\times \beta$ with $\alpha\in\Irr(A)$ and $\beta\in\Irr(B)$.
Let $D=A\cap N$. Notice that $\psi_{E}$ is zero off $DB$ by definition and that $\psi_{DB}=\theta_{D}\times\chi_{B}$.
Hence, we have
$$[\psi_{E},\xi]=\frac{|DB|}{|E|}[\psi_{DB},\xi_{DB}]=\frac{|D|}{|A|}[\theta_{D}\times\chi_{B},\alpha_{D}\times \beta],$$
and it follows that $|A:D|[\psi_{E},\xi]\in \mathbb{Z}$. In particular, $[\psi_{E},\xi]\in \mathbb{Q}$.

For any $x\in G$, we put
$$\omega(x)=\frac{\mu(1)|G:C_G(Nx)|\chi(x)}{\chi(1)},$$
which is an algebraic integer by Lemma \ref{alg-int}.
Notice that if $x\in B$, then $NA \le C_{G}(Nx)$, and we have
\begin{eqnarray*}
[\psi_{E},\xi] &=& \frac{|D|}{|A|}[\theta_{D}\times\chi_{B},\alpha_{D}\times \beta]\\
&=& \frac{|D|}{|A|}[\theta_{D},\alpha_{D}][\chi_{B},\beta]\\
&=&\frac{|D|}{|A||B|} [\theta_{D},\alpha_{D}] \sum\limits_{x\in B}\chi(x)\overline{\beta(x)}\\
&=& \frac{|D|}{|A||B|} [\theta_{D},\alpha_{D}]
\sum\limits_{x\in B}\frac{\chi(1)}{\mu(1)|G:C_G(Nx)|}\omega(x)\overline{\beta(x)}\\
&=& \frac{\chi(1)|N|}{\mu(1)|G||B|}[\theta_{D},\alpha_{D}]\sum\limits_{x\in B}|C_{G}(Nx):NA|\omega(x)\overline{\beta(x)},
\end{eqnarray*}
This shows that $\frac{\mu(1)|G:N|}{\chi(1)}|B|[\psi_{E},\xi]$ is an algebraic integer,
and since we have seen that this number is also rational, so it must be an integer.
Now, we see that $\frac{\mu(1)|G:N|}{\chi(1)}$ is a $\pi'$-number,
and that $|A:D|$ is a $\pi$-number.
It follows that $[\psi_{E},\xi]\in \mathbb{Z}$, as wanted.

(b) Let $G^0=\{x\in G|x_\pi\in N\}$.
Then $G^0$ is the full preimage in $G$ of the set ${\bar G}^0$ of
the $\pi'$-elements in $\bar G=G/N$.
For each $x\in G^0$, we have $Nx=Nx_{\pi'}$ and thus $\chi(x_{\pi'})=\chi(Nx_{\pi'})=\chi(\bar x)$.
Also, observe that $x_{\pi'}\in C_G(N)$, so $y_{\pi}\in N$ for each $y\in Nx$,
and it follows that
$$\sum_{y\in Nx}|\theta(y_\pi)|^2=\sum_{a\in N}|\theta(a)|^2=|N|[\theta,\theta]=|N|.$$
Now $\psi(1)=\theta(1)\chi(1)>0$, and we obtain
\begin{eqnarray*}
[\psi,\psi] &=& \frac{1}{|G|}\sum\limits_{x\in G}|\psi(x)|^{2} \\
               &=& \frac{1}{|G|}\sum\limits_{x\in G^0}|\theta(x_{\pi})|^2|\chi(x_{\pi'})|^{2} \\
               &=& \frac{1}{|G|}\sum\limits_{\bar x\in {\bar G}^0}|\chi(\bar x)|^{2}\sum_{y\in Nx}|\theta(y_{\pi})|^2  \\
               &=& \frac{|N|}{|G|}\sum\limits_{\bar x\in {\bar G}^0}|\chi(\bar x)|^{2}
               \le [\chi,\chi]=1.
\end{eqnarray*}
This proves that $\psi\in\Irr(G)$. In this case, since $\psi_N=\chi(1)\theta$,
we have
$$(\psi(1)/\theta(1))_\pi=\chi(1)_\pi=|G:N|_\pi,$$
and thus $\psi\in\rdz_\pi(G|\theta)$.

(c) In this case, we see that both $\theta$ and $\mu$ are linear, and thus for each $g\in G$ with $g_\pi\in N$,
we have $\chi(g)=\mu(g_\pi)\chi(g_{\pi'})$, so $|\chi(g)|=|\chi(g_{\pi'})|$.
As in (b), we get
\begin{eqnarray*}
[\psi,\psi] &=& \frac{1}{|G|}\sum\limits_{g\in G}|\psi(g)|^{2} \\
               &=& \frac{1}{|G|}\sum\limits_{g\in G^0}|\theta(g_{\pi})|^2|\chi(g_{\pi'})|^{2} \\
               &=& \frac{1}{|G|}\sum\limits_{g\in G^0}|\chi(g)|^{2}\le [\chi,\chi]=1,
\end{eqnarray*}
and since $\psi(1)>0$, equality holds everywhere.
In particular, $\psi\in\Irr(G)$ and $\chi$ vanishes on elements of $G-G^0$.
By the same argument used in (b), we deduce that $\psi\in\rdz_\pi(G|\theta)$,
so the map $\chi\mapsto\psi$ is well defined.

What remains is to show that this map is bijective.
To do this, we can interchange the roles of $\mu$ and $\theta$ to define the map
$\rdz_\pi(G|\theta)\to \rdz_\pi(G|\mu),\psi\mapsto\chi'$ by setting
$$\chi'(g)=\left\{
\begin{array}{cc}
\mu(g_{\pi})\psi(g_{\pi'}), & \text{if}\; g_{\pi}\in N,\\
0,& \text{otherwise}.
\end{array}
\right.$$
For each $g\in G$, we see that $\chi(g)=0=\chi'(g)$ if $g_\pi\notin N$,
and in the remaining case where $g_\pi\in N$, by definition we have
$$\chi'(g)=\mu(g_{\pi})\psi(g_{\pi'})=\mu(g_\pi)\theta(1)\chi(g_{\pi'})=\mu(g_\pi)\chi(g_{\pi'})=\chi(g).$$
This proves that $\chi=\chi'$, and thus the map $\chi\mapsto\psi$ is invertible.
The result now follows.
\end{proof}

As an immediate consequence, we can prove the following, which is Theorem D from the introduction.

\begin{cor}\label{D}
Let $(G,N,\theta)$ be a character triple. If the order of $[\theta]_{G/N}\in H^2(G/N,\mathbb{C}^\times)$ is a $\pi$-number
for a prime set $\pi$, then $|\dz_\pi(G/N)|=|\rdz_\pi(G|\theta)|$.
\end{cor}
\begin{proof}
By Lemma \ref{tri}, we can assume that $N$ is a central $\pi$-subgroup of $G$,
and the result follows by Theorem \ref{re}(c) (with $\mu=1_N$).
\end{proof}

\section{Proof of Theorem C}

We are now ready to prove Theorem C in the introduction, which we restate here for convenience.

\begin{thm}\label{qua-cor}
Let $(G,N,\theta)$ be a character triple,
and suppose that $\tilde\theta$ is a $\pi$-quasi extension of $\theta$ to $G$,
where $\pi$ is a set of primes.
Then the map $\chi\mapsto \tilde{\theta}\chi$ defines a bijection from $\dz_\pi(G/N)$ onto $\rdz_\pi(G|\theta)$.
\end{thm}
\begin{proof}
Fix $\chi\in\dz_\pi(G/N)$, and we first show that $\tilde\theta\chi$ is a generalized character of $G$.
By Brauer's theorem (Theorem 8.4 of \cite{I1976}),
it suffices to prove that for each nilpotent subgroup $E\le G$,
$(\tilde\theta\chi)_{NE}$ is a generalized character of $NE$.
To do this, writing $H=NE$, we must show that $[(\tilde\theta\chi)_H,\xi]\in\mathbb{Z}$ for each $\xi\in\Irr(H)$.

Observe that $H/N\cong E/N\cap E$ is nilpotent,
so by Theorem 6.22 of \cite{I1976}, there exist a subgroup $M$ and some character $\eta\in\Irr(M)$
such that $N\le M\le H$, $\xi=\eta^H$ and $\eta_N\in\Irr(N)$.
In this situation, since $M/N$ is also nilpotent, we can write $M/N=A/N\times B/N$,
where $A/N$ is a $\pi$-group and $B/N$ is a $\pi'$-group,
and since $\chi$ has $\pi$-defect zero, it follows that $\chi$ vanishes on $M-B$.
Now, we have
\begin{eqnarray*}
[(\tilde\theta\chi)_H,\xi] &=& [(\tilde\theta\chi)_M,\eta] \\
                                   &=& \frac{1}{|M|}\sum_{x\in B}\tilde\theta(x)\chi(x)\overline{\eta(x)} \\
                                   &=& \frac{1}{|M|}\sum_{\bar x\in B/N}\chi(\bar x) \sum_{y\in Nx} \tilde\theta(y)\overline{\eta(y)}.
\end{eqnarray*}
Note that $B/N$ is a $\pi'$-group, so by definition $\tilde\theta_B$ is an extension of $\theta$,
and since $\eta_B$ is clearly irreducible, it follows by \cite[Lemma 8.14]{I1976} that
$\sum_{y\in Nx} \tilde\theta(y)\overline{\eta(y)}=0$ unless $\eta_N=\theta$.
Thus it suffices to consider those $\xi\in\Irr(H)$ with $\eta_N=\theta$, and in this case,
we see that $\sum_{y\in Nx} |\eta(y)|^2=|N|$, and that both $\tilde\theta_B$ and $\eta_B$ extend $\theta$.
By Gallagher's theorem, we can write $\tilde\theta_B=\lambda\eta_B$ for some character $\lambda\in\Irr(B/N)$,
and we obtain
\begin{eqnarray*}
 [(\tilde\theta\chi)_H,\xi] &=& \frac{1}{|M|}\sum_{\bar x\in B/N}\chi(\bar x) \sum_{y\in Nx} |\eta(y)|^2\lambda(y) \\
                                    &=&  \frac{|N|}{|M|}\sum_{\bar x\in B/N}\chi(\bar x) \lambda(\bar x)\\
                                    &=&  \frac{1}{|M/B|}[\chi_{B/N},\lambda].
\end{eqnarray*}
So  $|M/B| [(\tilde\theta\chi)_H,\xi]$ is an integer,
and in particular, $[(\tilde\theta\chi)_H,\xi]$ is a rational number.

Furthermore, for each $\bar x\in B/N$, we know that $A/N\le C_{G/N}(\bar x)$ and that
$$\omega(\bar x)=\frac{|G/N:C_{G/N}(\bar x)|\chi(\bar x)}{\chi(1)}$$
is always an algebraic integer (see the discussion preceding Theorem 3.7 of \cite{I1976}). Thus,
\begin{eqnarray*}
 [(\tilde\theta\chi)_H,\xi] &=&  \frac{|N|}{|M|}\sum_{\bar x\in B/N}\chi(\bar x) \lambda(\bar x)\\
                                    &=&  \frac{1}{|M/N|}\sum_{\bar x\in B/N}\frac{\chi(1)\omega(\bar x)}{|G/N:C_{G/N}(\bar x)|} \lambda(\bar x)\\
                                    &=& \frac{\chi(1)}{|G/N||B/N|}\sum_{\bar x\in B/N} |C_{G/N}(\bar x):A/N|
                                    \omega(\bar x)\lambda(\bar x),
\end{eqnarray*}
which implies that $\frac{|G/N||B/N|}{\chi(1)}[(\tilde\theta\chi)_H,\xi]$ is also an algebraic integer
and hence a rational integer.
Moreover, notice that $|M/B|=|A/N|$ is a $\pi$-number while $\frac{|G/N||B/N|}{\chi(1)}$ is a $\pi'$-number,
and thus $[(\tilde\theta\chi)_H,\xi]$ is an integer.
This proves that $\tilde\theta\chi$ is a generalized character of $G$, as wanted.

Next, we show that $\tilde\theta\chi\in\Irr(G)$.
As before, we write $G^0=\{x\in G|x_{\pi}\in N\}$, and for each $x\in G^0$,
we have $Nx=Nx_{\pi'}$, so that $N\langle x\rangle/N$ is a $\pi'$-group.
In this case, by definition we know that $\tilde\theta_{N\langle x\rangle}$ extends $\theta$,
and by \cite[Lemma 8.14]{I1976} again, it follows that $\sum_{y\in Nx}|\tilde\theta(y)|^2=|N|$. Thus,
for each $\chi,\psi\in\dz_\pi(G/N)$, we conclude that
\begin{eqnarray*}
 [\tilde\theta\chi,\tilde\theta\psi] &=&  \frac{1}{|G|}\sum_{x\in G^0}|\tilde\theta(x)|^2 \chi(x) \overline{\psi(x)}\\
                                    &=&  \frac{1}{|G|}\sum_{\bar x\in G/N} \chi(\bar x) \overline{\psi(\bar x)}
                                    \sum_{y\in Nx}|\tilde\theta(y)|^2\\
                                    &=& \frac{1}{|G/N|}\sum_{\bar x\in G/N}\chi(\bar x) \overline{\psi(\bar x)}\\
                                    &=& [\chi,\psi].
\end{eqnarray*}
Now, taking $\chi=\psi$, we see that $[\tilde\theta\chi,\tilde\theta\chi]=1$, which means that $\tilde\theta\chi\in\Irr(G)$.
Also, if $\chi\neq \psi$, we know that $[\tilde\theta\chi,\tilde\theta\psi]=0$, that is, $\tilde\theta\chi\neq \tilde\theta\psi$.
This establishes the injectivity of the map $\chi\mapsto\tilde\theta\chi$,
and the surjectivity follows by Theorem D (see Corollary \ref{D}). The proof is now complete.
\end{proof}

\section*{Acknowledgement}
The authors would like to thank Professor Jiping Zhang for his guidance and encouragement.
This work was supported by the NSF of China (No. 12171289).



\end{document}